\documentclass[10pt]{amsart}
\usepackage{amsmath,amssymb,amsthm}
\usepackage{enumerate}             %  Enumerate items with '(i)' instead of '(1)'
\usepackage{cite}                  %  Improves LaTeX cite command.
\usepackage[margin=1in]{geometry}  %  Changes margins of the page.
\usepackage{hyperref}
\usepackage[numbered]{bookmark}
\usepackage{graphicx}
\usepackage{wrapfig,floatrow}
\usepackage{xstring}
\usepackage{etoolbox}
\usepackage{pgf,pgffor}
\usepackage{tikz}
\usetikzlibrary{patterns}
\usetikzlibrary{decorations.pathreplacing}
\usetikzlibrary{arrows}

\renewcommand{\ge}{\geqslant}
\renewcommand{\le}{\leqslant}

\newlength{\bulletRaiseLen}
\setbox1=\hbox{$\bullet$}\setbox2=\hbox{\tiny$\bullet$}
\newtheorem{theorem}{Theorem}[section]
\newtheorem*{theorem-nn}{Theorem}

\newtheorem{lemma}[theorem]{Lemma}
\newtheorem{proposition}[theorem]{Proposition}

\theoremstyle{remark}
\newtheorem{remark}[theorem]{Remark}

\DeclareRobustCommand{\warning}{\fontencoding{U}\fontfamily{futs}\selectfont\char 66\relax}
\DeclareRobustCommand{\wsign}{\raisebox{0.4mm}{\warning}\,}

\newcommand{\+}{\mkern1.5mu}
\renewcommand{\emph}[1]{{\bf #1}}
\newcommand{\Lem}{\mathrm{Lem}\,}
\newcommand{\dist}{\mathrm{dist}}
\newcommand{\tileable}{\mathcal{T}}
\newcommand{\es}{\varnothing}
\newcommand{\acts}{\curvearrowright}
\newcommand{\mff}{\mathfrak{F}}

\makeatletter
\def\oer{\@ifnextchar[{\oer@i}{\oer@i[]}}
\def\oer@i[#1]#2{\mathsf{E}^{#1}_{#2}}

\def\rgap{\@ifnextchar[{\rgap@i}{\rgap@i[ ]}}
\def\rgap@i[#1]{\mathrm{ga}\vec{\mathrm{p}}_{#1}}
\makeatother

\begin{document}
\title{Cross sections of Borel flows\\ with restrictions on the distance set}

\author{Konstantin Slutsky}
\address{Department of Mathematics, Statistics, and Computer Science\\
University of Illinois at Chicago\\
322 Science and Engineering Offices (M/C 249)\\
851 S. Morgan Street\\
Chicago, IL 60607-7045}
\email{kslutsky@gmail.com}
\begin{abstract}
  Given a set of positive reals, we provide a necessary and sufficient condition for a free Borel
  flow to admit a cross section with all distances between adjacent points coming from this set.
\end{abstract}

\maketitle
%%%%%%%%%%

\section{Introduction}
\label{sec:introduction}

This paper completes the study initiated in \cite{slutsky2015}, where a Borel version of D.~Rudolph's
\cite{rudolph_two-valued_1976} two-step suspension flow representation is given.  The main result of
the current work is a criterion for a given set \( S \subseteq \mathbb{R}^{>0} \) and a free Borel
flow \( \mff \) to admit a cross section with distances between adjacent points belonging to
\( S \).

A cross section for a flow leads to a representation of the flow as a flow under a function  (see
Figure~\ref{fig:flow-under-function} and 
\cite[Section~7]{MR2963410}, \cite[Section~2]{slutsky2015}).  Properties of the flow are
reflected in the properties of the base automorphism, but details of their interplay are obscured by
the gap function.  To get a more transparent connection between the flow and the base automorphism,
it is often desirable to impose restrictions on the distances between adjacent points in the cross
section.

\begin{wrapfigure}[11]{L}{5.5cm}
%\begin{figure}[htb]
  \centering
  \begin{tikzpicture}
    \draw[thick] plot [smooth, tension=1] coordinates { (0,2) (1.5,2.5) (3,2) (4.5,2.2)};
    \draw[thick] (0,0) -- (4.5,0);
    \filldraw (1.7,1) circle (1pt);
    \draw [->,>=stealth] (1.7,1) -- (1.7,2.47);
    \draw [dotted] (1.7,1) -- (1.7,0);
    \filldraw (1.7,0) circle (1pt);
    \filldraw (3.3,0) circle (1pt);
    \draw [->,>=stealth] (3.3,0) -- (3.3, 1.0);
    \draw (-0.1,-0.26) node[anchor=west] {cross section \( \mathcal{C} \)};
    \draw (3.35, 1.7) node {gap function \( f \)};
    \draw[->, >=stealth] (1.7, 0.02) [out=35, in=145] to  (3.27, 0.02);
    \draw (2.5, 0.47) node {\( \phi_{\mathcal{C}} \)};
  \end{tikzpicture}
  \caption{}
  \label{fig:flow-under-function}
  % \end{figure}
\end{wrapfigure}

Of particular importance here are cross sections with only two distinct distances between adjacent
points.  Their existence in the sense of ergodic theory was proved in
\cite{rudolph_two-valued_1976}, and they were used to resolve a problem of Sinai on equivalence of
two definitions of \( K \)-flows.  Further improvement of Rudolph's construction by U.~Krengel
\cite{MR0435354} gave a version of Dye's Theorem for ergodic flows.  The Borel version of these
results obtained in \cite{slutsky2015}, gives a short proof of the analog of the R.~Dougherty,
S.~Jackson, A.~S.~Kechris \cite{MR1149121} classification of Borel flows up to Lebesgue orbit
equivalence (see Theorem 10.4 in \cite{slutsky2015} and Theorem 9.1 in \cite{2015arXiv150400958S}).
We hope that constructions of cross sections in the present paper will be useful in further
explorations of connections between properties of flows and automorphisms they induce on cross
sections.

A Borel flow is a Borel measurable action of \( \mathbb{R} \) on a standard Borel space
\( \Omega \).  Actions are denoted additively: \( \omega + r \) denotes the action of
\( r \in \mathbb{R} \) upon \( \omega \in \Omega \).  A cross section for a flow
\( \mathbb{R} \acts \Omega \) is a Borel set \( \mathcal{C} \subseteq \Omega \) that intersects
every orbit in a non-empty lacunary set (``lacunarity'' means existence of
\( c \in \mathbb{R}^{>0} \) such that for any \( x \in \mathcal{C} \) and
\( r \in \mathbb{R}^{>0} \) inclusion \( x + r \in \mathcal{C}\setminus\{x\} \) implies
\( r \ge c \)).  Existence of cross sections was first shown by
V.~M.~Wagh~\cite{wagh_descriptive_1988}, improving upon earlier works of W.~Ambrose and S.~Kakutani
\cite{ambrose_representation_1941}, \cite{MR0005800}.  When the flow is free, every orbit becomes an
affine copy of \( \mathbb{R} \), and any translation invariant notion can therefore be transferred
from \( \mathbb{R} \) onto orbits of the flow.  In particular, given two points
\( \omega_{1}, \omega_{2} \in \Omega \) within the same orbit one may naturally define the distance
\( \dist(\omega_{1}, \omega_{2}) \) between them.

We always assume that our flows are free and cross sections are ``bi-infinite'' on each orbit
\textemdash{} if \( \mathcal{C} \subseteq \Omega \) is a cross section, then every
\( x \in \mathcal{C} \) has a successor and a predecessor among elements of \( \mathcal{C} \) from
the same orbit.  This allows us to endow \( \mathcal{C} \) with an induced automorphism
\( \phi_{\mathcal{C}} : \mathcal{C} \to \mathcal{C} \) which sends a point to the next one.  We also
let \( \rgap[\mathcal{C}] : \mathcal{C} \to \mathbb{R}^{>0} \) to denote the gap function which
measures distance to the next point:
\( \rgap[\mathcal{C}](x) = \dist\bigl(x, \phi_{\mathcal{C}}(x)\bigr) \).

Given a non-empty set \( S \subseteq \mathbb{R}^{>0} \), we say that a cross section
\( \mathcal{C} \) is \( S \)-regular if \( \rgap[\mathcal{C}](x) \in S \) for all
\( x \in \mathcal{C} \), i.e., if the distances between adjacent points in \( \mathcal{C} \)
belong to \( S \).  The following was proved in \cite{slutsky2015}: Given any two positive
rationally independent reals \( \alpha, \beta \in \mathbb{R}^{>0}\), any free Borel flow admits an
\( \{\alpha,\beta\} \)-regular cross section.  In this paper we push the methods of
\cite{slutsky2015} a little further and for a given \( S \subseteq \mathbb{R}^{>0} \) give a
criterion for a flow to admit an \( S \)-regular cross section.  Recall that a flow is said to be
sparse if it admits a cross section with gaps ``bi-infinitely'' unbounded on each orbit.  A subgroup
of \( \mathbb{R} \) generated by \( S \) is denoted by \( \langle S \+ \rangle \).

\begin{theorem-nn}[see Theorem \ref{thm:main-theorem}]
  Let \( \mff \) be a free Borel flow on a standard Borel space \( X \) and let
  \( S \subseteq \mathbb{R}^{>0} \) be a set bounded away from zero.
  \begin{enumerate}[(I)]
  \item\label{item:lambda-regular-mt}   Assume  \(   \langle S \+ \rangle   =  \lambda\mathbb{Z}   \),
    \( \lambda > 0 \). The flow \( \mff \) admits an \( S \)-regular cross section if and only if it
    admits a \( \{\lambda\} \)-regular cross section.
  \item\label{item:dense-regular-mt} Assume \( \langle S \+ \rangle\) is dense in \( \mathbb{R} \), but
    \( \langle S \cap [0, n] \rangle = \lambda_{n} \mathbb{Z} \), \( \lambda_{n} \ge 0 \), for all
    natural \( n \in \mathbb{N} \) (we take \( \lambda_{n} = 0 \) if \( S \cap [0,n] \) is empty).
    The flow \( \mff \) admits an \( S \)-regular cross section if and only if the phase space
    \( X \) can be partitioned into \( \mff \)-invariant Borel pieces (some of which may be empty)
    \[ X = \Bigl(\bigsqcup_{i=0}^{\infty} X_{i}\Bigr) \sqcup X_{\infty} \]
    such that \( \mff|_{X_{\infty}} \) is sparse and \( \mff|_{X_{i}} \) admits a
    \( \{\lambda_{i}\} \)-regular cross section.
  \item\label{item:bounded-dense-regular-mt} Assume there \( n \in \mathbb{N} \) such that
    \( \langle S \cap [0,n] \rangle \) is dense in \( \mathbb{R} \).  Any flow admits an
    \( S \)-regular cross section.
  \end{enumerate}
\end{theorem-nn}

To further explore item \eqref{item:lambda-regular-mt}, it is, perhaps, helpful to recall a
criterion of W.~Ambrose \cite{ambrose_representation_1941} (see also \cite[Proposition
2.5]{slutsky2015}) for a flow to admit a \( \{\lambda\} \)-regular cross section.
\begin{proposition}
  \label{prop:Ambrose}
  A free Borel flow \( \mff \) on \( X \) admits a cross section with all gaps of  size
  \( \lambda > 0 \) if and only if there is a Borel function
  \( f : X \to \mathbb{C}\setminus\{0\} \) such that
  \[ f(x + r) = e^{\textstyle \frac{2 \pi i r}{\lambda}} f(x)\ \textrm{ for all } x \in X \textrm{ and } r \in
  \mathbb{R}.\]
\end{proposition}

The paper is concluded with an example of a flow which shows that the condition in item
\eqref{item:dense-regular} of the main theorem is not vacuous.

\subsection{Notations}
\label{sec:notations}

The following notations are used throughout the paper.  For a set \( S \subseteq \mathbb{R}^{>0} \)
and a cross section \( \mathcal{C} \), \( \oer[S]{\mathcal{C}} \) denotes the equivalence relations
defined by
\[ x\, \oer[S]{\mathcal{C}}\, y \iff \exists n \in \mathbb{N}\ \phi_{\mathcal{C}}^{n}(x) = y
\textrm{ and } \rgap[\mathcal{C}]\bigl(\phi_{\mathcal{C}}(x)^{k}\bigr) \in S \textrm{ for all }
0\le k < n,\]
or the same condition with roles of \( x \) and \( y \) switched.  In plain words, \( x\,
\oer[S]{\mathcal{C}}\, y \) if all the gaps, when going from \( x \) to \( y \) in \( \mathcal{C} \),
belong to \( S \).  To say that \( \mathcal{C} \) is \( S \)-regular is the same as to say that \(
\oer[S]{\mathcal{C}} \) coincides with the orbit equivalence relation induced on \( \mathcal{C} \).  We
also let \( \oer[\le K]{\mathcal{C}} \) denote the relation \( \oer[{[}0,K{]}]{\mathcal{C}} \).

A set \( S \subseteq \mathbb{R} \) is said to be \( \epsilon \)-dense in an interval
\( I \subseteq \mathbb{R} \) if for every open sub-interval \( J \subseteq I \) of length
\( \epsilon \) the intersection \( J \cap S \) is non-empty.  An \( \epsilon \)-neighborhood
\( (x - \epsilon, x + \epsilon) \) of \( x \in \mathbb{R} \) is denoted by
\( \mathcal{U}_{\epsilon}(x) \).  For a set \( S \subseteq \mathbb{R}^{>0} \), the semigroup
generated by \( S \) is denoted by \( \tileable(S) \):
\[ \mathcal{T}(S) = \Bigl\{\, \sum_{k=1}^{n}s_{k} \,\Bigm|\, n \ge 1,\ s_{k} \in S \,\Bigr\}. \]
The group generated by \( S \) is, as usually, denoted by \( \langle S \+ \rangle \).  We say that a
set \( S \subseteq \mathbb{R}^{\ge 0} \) is asymptotically dense in \( \mathbb{R} \) if for every
\( \epsilon > 0 \) there is \( K \ge 0 \) such that \( S \) is \( \epsilon \)-dense in
\( [K,\infty) \).

\section{Regular cross sections of sparse flows}
\label{sec:rudolphs-method}

\begin{lemma}
  \label{lem:dense-subgroup-asymp-dense-semigroup}
  Let \( S \subseteq \mathbb{R}^{>0} \) be a non-empty subset.  The following are equivalent.
  \begin{enumerate}[(i)]
  \item\label{item:dense-subgroup} \( \langle S \+ \rangle \) is dense in \( \mathbb{R} \).
  \item\label{item:finite-subset-eps-dense} For every \( \epsilon > 0 \) there exists a finite
    \( F \subseteq S \) and \( K \in \mathbb{R}^{\ge 0} \) such that \( \tileable(F) \) is
    \( \epsilon \)-dense in \( [K, \infty) \).
  \item\label{item:asymptotically-dense} \( \tileable(S) \) is asymptotically dense in
    \( \mathbb{R} \).
  \end{enumerate}
\end{lemma}
\begin{proof}
  Implications \eqref{item:asymptotically-dense} \( \implies \) \eqref{item:dense-subgroup} and
  \eqref{item:finite-subset-eps-dense} \( \implies \) \eqref{item:asymptotically-dense} are
  obvious.  We prove \eqref{item:dense-subgroup} \( \implies \)
  \eqref{item:finite-subset-eps-dense}.
  
  Suppose \( S \) generates a dense subgroup.  Pick an element \( \tilde{s} \in S \) and an
  \( 0 < \epsilon < \tilde{s}/2 \).  Select a finite set
  \( \tilde{F} \subseteq \langle S \+ \rangle \) such that \( \tilde{F} \subseteq [0,\tilde{s}] \)
  and \( \tilde{F} \) is \( \epsilon/2 \)-dense in \( [0,\tilde{s}] \).  Let
  \( F = \{s_{k}\}_{k=0}^{n} \subseteq S \) be a finite set such that \( s_{0} = \tilde{s} \) and
  any \( f \in \tilde{F} \) is of the form \( f = \sum_{k=0}^{n}a_{f,k}s_{k} \) for some
  \( a_{f,k} \in \mathbb{Z} \), \( k \le n \).  Such coefficients \( a_{f,k} \) may not be unique,
  for each \( f \in \tilde{F}\) we fix one such decomposition.  Let \( M = \max_{f,k}|a_{f,k}| \)
  and take \( K = M \cdot \sum_{k=0}^{n} s_{k} \).  We claim that \( \mathcal{T}(F) \) is
  \( \epsilon \)-dense in \( [K,\infty) \).  Indeed, take
  \( \mathcal{U}_{\epsilon/2}(y) \subseteq [K, \infty] \), and pick \( r \in \mathbb{N} \) such
  that \( y - r \tilde{s} \in [K, K+\tilde{s}) \).  Since
  \[ \textrm{either } (y, y + \epsilon/2) - r\tilde{s} \subseteq [K, K+\tilde{s}) \textrm{ or } (y-
  \epsilon/2, y) - r\tilde{s} \subseteq [K, K+\tilde{s}), \]
  one may find
  \( f \in \tilde{F} \) such that \( f + r\tilde{s} + K  \in \mathcal{U}_{\epsilon/2}(y)  \)  Since
  \( f + K \in \tileable(F) \) and \( r\tilde{s} \in \tileable(F) \), we get \( f + r \tilde{s} + K
  \in \tileable(S) \), and so \( \tileable(F) \) is \( \epsilon \)-dense in \( [K,\infty) \).
\end{proof}

\begin{theorem}
  \label{thm:regular-S-sections-for-sparse-flows}
  Let \( S \subseteq \mathbb{R}^{>0} \) be a non-empty set bounded away from zero.  If
  \( \langle S \+ \rangle \) is dense in \( \mathbb{R} \), then any sparse flow admits an
  \( S \)-regular cross section.
\end{theorem}

\begin{proof}
  Let \( \mff \) be a free sparse Borel flow on a standard Borel space \( \Omega \), and let
  \( S \subseteq \mathbb{R}^{>0} \) be such that \( \langle S \+ \rangle \) is dense in
  \( \mathbb{R} \).  It is easy to see that if \( \langle S \+ \rangle \) is dense in \( \mathbb{R} \),
  then there is a countable (possibly finite) subset \( S' \subseteq S \) which also generates a dense
  subgroup of \( \mathbb{R} \), and we may therefore assume without loss of generality that \( S \)
  is countable.
  
  By Lemma \ref{lem:dense-subgroup-asymp-dense-semigroup}, the semigroup \( \mathcal{T}(S) \) is
  asymptotically dense in \( \mathbb{R}^{>0} \), and so there exists a function
  \( \xi : \mathbb{R}^{>0} \to \mathbb{R} \) such that \( x + \xi(x) \in \mathcal{T}(S) \) and
  \( \xi(x) \to 0 \) as \( x \to +\infty \).  Such a function can be picked Borel.
  Set \( \epsilon_{n} = 2^{-n-1}/3 \), and let \( (K_{n})_{n=0}^{\infty} \) be an
  increasing sequence, \( K_{n+1} > K_{n} + 1 \), such that \( |\xi(x)| < \epsilon_{n} \) for all
  \( x \ge K_{n} - 2 \).

  We construct cross sections \( \mathcal{C}_{n} \), Borel functions \( h_{n+1} : \mathcal{C}_{n}
  \to (-\epsilon_{n}, \epsilon_{n}) \), and finite Borel equivalence relations \( \oer{n} \) on \(
  \mathcal{C}_{n} \) which will satisfy the following list of properties.
  \begin{enumerate}
  \item\label{item:E0-is-trivial} The relation \( \oer{0} \) on \( \mathcal{C}_{0} \) is the trivial
    equivalence relation: \( x\, \oer{0}\, y \) if and only if \( x = y \).
  \item\label{item:Cn-sparse} \( \mathcal{C}_{n} \) is a sparse cross section for every \( n \) and
    \( \rgap[\mathcal{C}_{n}](x) \ge 1 \) for all \( x \in \mathcal{C}_{n} \).
  \item\label{item:Cn+1-is-shifted-copy-of-Cn} \( \mathcal{C}_{n+1} = \mathcal{C}_{n} + h_{n+1} \),
    i.e.,
    \[ \mathcal{C}_{n+1} = \bigl\{ x + h_{n+1}(x) \bigm| x \in \mathcal{C}_{n} \bigr\}. \]
  \item\label{item:hn-are-constant-on-En-classes} \( h_{n+1} \) is constant on \( \oer{n}
    \)-classes: \( x\, \oer{n}\, y \implies h_{n+1}(x) = h_{n+1}(y) \).
  \item\label{item:En-classes-are-TS-regular} \( \oer{n} \)-classes are
    \( \mathcal{T}(S) \)-regular: \( \oer{n} \subseteq \oer[\tileable(S)]{\mathcal{C}_{n}} \).
  \item\label{item:En+1-is-coarser-En} \( \oer{n+1} \) is coarser than \( \oer{n} \):
    \[ x\, \oer{n}\, y \implies \bigl(x + h_{n+1}(x)\bigr)\, \oer{n+1} \bigl(y +
    h_{n+1}(y)\bigr). \]
  \item\label{item:En-classes-are-far-from-each-other} Distinct \( \oer{n} \)-classes are far from
    each other: if \( x, y \in \mathcal{C}_{n} \) belong to the same orbit and are not
    \( \oer{n} \)-equivalent, then \( \dist(x,y) > K_{n}-1 \).
  \item\label{item:En-classes-are-Kn-chained} If \( x\, \oer{n}\, \phi_{\mathcal{C}_{n}}(x) \), then
    \( \dist\bigl(x, \phi_{\mathcal{C}_{n}}(x)\bigr) \le K_{n}+1 \).
  \end{enumerate}

  Let us first finish the proof under the assumption that such cross sections have been manufactured.
  Set \( f_{n,n+1} : \mathcal{C}_{n} \to \mathcal{C}_{n+1} \) to be the map
  \( f_{n,n+1}(x) = x + h_{n+1}(x) \) and define \( f_{m,n} : \mathcal{C}_{m} \to \mathcal{C}_{n} \)
  for \( m \le n \) to be
  \[ f_{m,n} = f_{n-1,n} \circ f_{n-2, n-1} \circ \cdots \circ f_{m,m+1} \]
  with the agreement that \( f_{m,m} : \mathcal{C}_{m} \to \mathcal{C}_{m} \) is the identity map.
  Since \( |h_{n}(x)| < \epsilon_{n-1} \), and since \( \rgap[\mathcal{C}_{n}](x) \ge 1 \) by
  \( \eqref{item:Cn-sparse} \), it follows that maps \( f_{n,n+1} \) are injective, and thus so are all
  the maps \( f_{m,n} \), \( m \le n \).  Since they are also surjective by
  \eqref{item:Cn+1-is-shifted-copy-of-Cn}, the maps \( f_{m,n} \) are Borel isomorphisms between
  \( \mathcal{C}_{m} \) and \( \mathcal{C}_{n} \).  In simple words, \( \mathcal{C}_{n} \) is
  obtained from \( \mathcal{C}_{m} \) by moving each point of \( \mathcal{C}_{m} \) by at most
  \( \sum_{i=m}^{n-1} \epsilon_{i} \) as prescribed by functions \( h_{i} \), \( m < i \le n\).  Let
  \[ H_{m} : \mathcal{C}_{m} \to
  \Bigl(-\mkern-9mu\sum_{i=m}\mkern-4mu\epsilon_{i},\sum_{i=m}\mkern-4mu\epsilon_{i}\Bigr) \qquad
  H_{m}(x) = \sum_{n=m}^{\infty} h_{n+1}\bigl( f_{m,n}(x) \bigr). \]
  be the ``total shift'' function.  
  Note that \( H_{m}(x) = H_{n}\bigl(f_{m,n}(x)\bigr) \) for any \( x \in \mathcal{C}_{m} \) and
  \( m \le n \).  The limit cross section \( \mathcal{C}_{\infty} \) is defined by
  \( \mathcal{C}_{\infty} = \mathcal{C}_{0} + H_{0} \), i.e.,
  \[ \mathcal{C}_{\infty} = \{ x + H_{0}(x) \mid x \in \mathcal{C}_{0} \}. \]
  Note also that \( \mathcal{C}_{\infty} = \{ x + H_{m}(x) \mid x \in \mathcal{C}_{m} \} \) for any
  \( m \in \mathbb{N} \), and the map \( x \mapsto x + H_{m}(x) \) is a bijection between
  \( \mathcal{C}_{m} \) and \( \mathcal{C}_{\infty} \).

  We claim that \( \mathcal{C}_{\infty} \) is a \( \mathcal{T}(S) \)-regular cross section.  It is
  clear that \( \mathcal{C}_{\infty} \) is a cross section.  Let
  \( y_{1}, y_{2} \in \mathcal{C}_{\infty} \), \( y_{1} \ne y_{2} \), be given and let \( m \) be so
  large that \( K_{m} > \dist(y_{1}, y_{2}) + 2 \).
 
  Pick \( z_{1}, z_{2} \in \mathcal{C}_{m} \) such that \( y_{i} = z_{i} + H_{m}(z_{i}) \).  Since
  \( H_{m}(z_{i}) \le 1/3 \), 
  \[ \dist(z_{1},z_{2}) \le \dist(y_{1},y_{2}) + 2/3 < K_{m} - 1, \]
  hence \( z_{1}\, \oer{m}\, z_{2} \) by \eqref{item:En-classes-are-far-from-each-other}, whence
  \eqref{item:En-classes-are-TS-regular} implies that \( \dist(z_{1}, z_{2}) \in \mathcal{T}(S) \),
  but by \eqref{item:hn-are-constant-on-En-classes}
  and \eqref{item:En+1-is-coarser-En}
  we get
  \( H_{m}(z_{1}) = H_{m}(z_{2}) \).  Therefore,
  \( \dist(y_{1},y_{2}) = \dist(z_{1},z_{2}) \in \mathcal{T}(S) \).  
  Thus, \( \mathcal{C}_{\infty} \) is a \( \tileable(S) \)-regular cross section.

  We now add some points to \( \mathcal{C}_{\infty} \) to make it \( S \)-regular.  Let
  \( S^{< \omega}_{\uparrow} \) be the set of all tuples \( (0, t_{1}, \ldots, t_{m}) \),
  \( t_{k} \in \mathbb{R} \), such that \( 0 < t_{1} < \cdots < t_{m} \), and
  \( t_{k+1} - t_{k} \in S \), for all \( k < m \), \( m \in \mathbb{N} \).  Fix a map
  \( \zeta : \mathcal{T}(S) \to S^{< \omega}_{\uparrow} \) such that for any
  \( t \in \mathcal{T}(S) \) one has \( t = t_{m} \) , where \( \zeta(t) = (t_{k})_{k=1}^{m} \).  In
  other words, \( \zeta(t) \) is a way to decompose an interval of length \( t \) into intervals of
  lengths in \( S \).  Let \( \mathcal{C} \) be given by
  \[ \mathcal{C} = \bigl\{ x + t \bigm| x \in \mathcal{C}_{\infty},\ t \textrm{ is one of the
    coordinates in } \zeta\bigl( \rgap[\mathcal{C}_{\infty}](x) \bigr) \bigr\}.  \]
  Since \( S \) is bounded away from zero, \( \mathcal{C} \) is a lacunary \( S \)-regular cross
  section.

  It remains to show how such \( \mathcal{C}_{n} \), \( \oer{n} \), and \( h_{n} \) can be
  constructed.  Let \( \mathcal{C}_{0} \) be a sparse cross section; by passing to a sub cross
  section we may assume that \( \rgap[\mathcal{C}_{0}](x) > K_{0} \) for all
  \( x \in \mathcal{C}_{0} \).  We take \( \oer{0} \) to be the trivial equivalence relation.

  Suppose we have constructed \( \mathcal{C}_{n} \), \( \oer{n} \), and
  \( h_{n} : \mathcal{C}_{n-1} \to (-\epsilon_{n-1}, \epsilon_{n-1}) \).  Consider the relation
  \( \oer[\le K_{n+1}]{\mathcal{C}_{n}} \) on \( \mathcal{C}_{n} \).
  By item \eqref{item:En-classes-are-Kn-chained} \( \oer[\le K_{n+1}]{\mathcal{C}_{n}} \) is coarser than
  \( \oer{n} \) (recall that \( K_{n+1} \ge K_{n} + 1 \)).  Since \( \mathcal{C}_{n} \) is sparse by \eqref{item:Cn-sparse}, each
  \( \oer[\le K_{n+1}]{\mathcal{C}_{n}} \)-class is finite and constitutes an interval in
  \( \mathcal{C}_{n} \).  Any \( \oer[\le K_{n+1}]{\mathcal{C}_{n}} \)-class consists of finitely
  many \( \oer{n} \)-classes.
  \begin{figure}[htb]
    \centering
    \begin{tikzpicture}
      \foreach \x in {0.0,0.7,1.3,1.9, 4.1,4.9,5.8, 7.8,8.6,9.3,10.2} {
        \filldraw[black] (\x, 0) circle (1pt);
      }
      \draw (0.78,0.25) node {\( x_{1} \)}
      (4.98,0.25) node {\( x_{2} \)}
      (9.38,0.25) node {\( x_{3} \)};
      \draw [
      decoration={
        brace,
        mirror,
        raise=0.3cm
      },
      decorate
      ] (-0.02,0) -- (1.92,0) node [pos=0.585,anchor=north,yshift=-0.4cm]
      {\( [x_{1}]_{\oer{n}} \)};
      \draw [
      decoration={
        brace,
        mirror,
        raise=0.3cm
      },
      decorate
      ] (4.08,0) -- (5.82,0) node [pos=0.6,anchor=north,yshift=-0.4cm]
      {\( [x_{2}]_{\oer{n}} \)};
      \draw [
      decoration={
        brace,
        mirror,
        raise=0.3cm
      },
      decorate
      ] (7.78,0) -- (10.22,0) node [pos=0.57,anchor=north,yshift=-0.4cm]
      {\( [x_{3}]_{\oer{n}} \)};
      \draw (3,0) node {\( d_{2} \)};
      \draw (6.8,0) node {\( d_{3} \)};
      \draw (5.1, -1.55) node[text width=8cm] {Each \( [x_{i}]_{\oer{n}} \)-class is shifted by at most
      \( \epsilon_{n} \) so as to make gaps between classes belong to \( \tileable(S) \)};
      \foreach \x in {0.0,0.7,1.3,1.9, 4.3,5.1,6.0, 7.6,8.4,9.1,10.0} {
        \filldraw[black] (\x, -2.3) circle (1pt);
      }
      \draw [
      decoration={
        brace,
        mirror,
        raise=0.2cm
      },
      decorate
      ] (2.0,-2.3) -- (4.20,-2.3) node [pos=0.5,anchor=north,yshift=-0.3cm]
      {\( \in \tileable(S) \)};
      \draw [
      decoration={
        brace,
        mirror,
        raise=0.2cm
      },
      decorate
      ] (6.1,-2.3) -- (7.5,-2.3) node [pos=0.5,anchor=north,yshift=-0.3cm]
      {\( \in \tileable(S) \)};
    \end{tikzpicture}
    \caption{Constructing \( \mathcal{C}_{n+1} \) from \( \mathcal{C}_{n} \).}
    \label{fig:constructing-Cn+1-from-Cn}
  \end{figure}
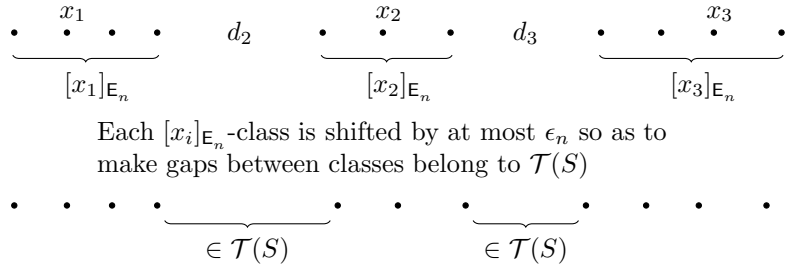
  Consider one such class and let \( x_{1}, \ldots, x_{m} \in \mathcal{C}_{n} \) be representatives
  of \( \oer{n} \)-classes in the \( \oer[\le K_{n+1}]{\mathcal{C}_{n}} \)-class:
  \begin{itemize}
  \item \( x_{1} < x_{2} < \cdots < x_{m} \);
  \item \( x_{i}\, \oer[\le K_{n+1}]{\mathcal{C}_{n}}\, x_{j} \);
  \item \( [x_{1}]_{\oer[\le K_{n+1}]{\mathcal{C}_{n}}} = \bigsqcup_{k=1}^{m} [x_{k}]_{\oer{n}} \).
  \end{itemize}
  Let \( d_{k} \), \( 2 \le k \le m \), be the gap between the \( k^{\mathrm{th}} \) and
  \( k-1^{\mathrm{st}} \) \( \oer{n} \)-classes:
  \[ d_{k} = \dist\bigl(\max[x_{k-1}]_{\oer{n}}, \min[x_{k}]_{\oer{n}} \bigr). \]
  By \eqref{item:En-classes-are-far-from-each-other}, \( d_{k} \ge K_{n} - 1 \), and
  therefore \( | \xi(d_{2})| < \epsilon_{n} \).  We let \( h_{n+1}(x) = 0 \) for
  \( x \in [x_{1}]_{\oer{n}} \) and \( h_{n+1}(x) = \xi(d_{2}) \) for
  \( x \in [x_{2}]_{\oer{n}} \).  By induction on \( k \) we set
  \[ h_{n+1}(x) = \xi\bigl( d_{k} - h_{n+1}(x_{k-1}) \bigr) \textrm{ for } x \in
    [x_{k}]_{\oer{n}}. \]
  In words, we shift \( \oer{n} \)-classes one by one by at most \( \epsilon_{n} \) to make
  distances between them elements of \( \tileable(S) \).  This can be done within each
  \( \oer[\le K_{n+1}]{\mathcal{C}_{n}} \)-class in a Borel way, thus defining a Borel map
  \( h_{n+1} : \mathcal{C}_{n} \to (-\epsilon_{n}, \epsilon_{n}) \).  Finally, we let
  \( \mathcal{C}_{n+1} = \mathcal{C}_{n} + h_{n+1} \), and
  \( \oer{n+1} = \oer[\le K_{n+1}]{\mathcal{C}_{n}} + h_{n+1} \), i.e.,
  \[ \bigl(x +h_{n+1}(x) \bigr) \oer{n+1}\bigl(y + h_{n+1}(y)\bigr) \quad \textrm{ if and only if}
  \quad x\, \oer[\le K_{n+1}]{\mathcal{C}_{n}}\, y. \]
  All the items (\ref{item:E0-is-trivial}-\ref{item:En-classes-are-Kn-chained}) are now easily
  verified.
\end{proof}

\section{Large regular blocks}
\label{sec:prop-freed-princ}

In this section we fix a positive real \( \upsilon \in \mathbb{R}^{>0} \) and a strictly monotone
sequence \( (t_{m})_{m=0}^{\infty} \) that converges to \( 0 \) and is such that \( \upsilon +
t_{0} > 0 \).  We set
\[ \tileable_{m} = \tileable(\{\upsilon + t_{0}, \ldots, \upsilon + t_{m}\}) \]
to denote the semigroup generate by \( \{ \upsilon+t_{i}: i \le m \} \), and also
\( \tileable_{m}^{*} = \upsilon + \tileable_{m} \).  In this section we do the necessary preparation
to show that every flow admits a cross section with arbitrarily large \( \bigcup_{m} \tileable_{m}
\)-regular blocks.

Let \( d_{1}, \ldots, d_{n} \) be a family of positive reals and let \( R_{i} \subseteq
\mathcal{U}_{\epsilon}(d_{i}) \) be non-empty subsets of the \( \epsilon \)-neighborhoods of \(
d_{i} \).  We let \( \mathcal{A}_{n} = \mathcal{A}\bigl(\epsilon, (d_{i})_{i=1}^{n},
(R_{i})_{i=1}^{n}\bigr) \) to denote the set of all \( z \in \mathcal{U}_{\epsilon}\bigl(
\sum_{i=1}^{n} d_{i} \bigr) \) for which there exist \( x_{i} \in R_{i} \) such that \( z =
\sum_{i=1}^{n} x_{i} \) and
\[ \bigl| \sum_{i=1}^{r} (d_{i}-x_{i}) \bigr| < \epsilon \textrm{ for all } r \le n. \]
When sequences \( (d_{i}) \) and \( (R_{i}) \) are constant, \( d := d_{i} \) and \( R := R_{i} \),
we use the notation \( \mathcal{A}_{n}(\epsilon, d, R) \).  For the geometric explanation of sets \(
\mathcal{A}_{n} \) we refer the reader to Subsection 6.2 of \cite{slutsky2015}.

For two non-zero reals \( a, b \in \mathbb{R} \) we let \( \gcd(a,b) \) to denote the largest
positive real \( c \) such that both \( a \) and \( b \) are integer multiples of \( c \).  If no
such real exists, i.e., if \( a \) and \( b \) are rationally independent, we set \( \gcd(a,b) = 0
\).

We need two lemmas from \cite{slutsky2015}, which we state below.

\begin{lemma}[see Lemma 6.7 in \cite{slutsky2015}] Sets \( \mathcal{A}_{n}(\epsilon,d,R) \) have the
  following additivity properties.
  \label{lem:additivity-of-An}
  \begin{enumerate}[(i)]
  \item\label{item:additivity-one-dim} If \( y_{i} \in R \), \( 1 \le i \le n \), are such that
    \[ \Bigl| nd - \sum_{i=1}^{n} y_{i} \Bigr| < \epsilon, \]
    then \( \sum_{i=1}^{n} y_{i} \in \mathcal{A}_{n}(\epsilon,d,R) \).
  \item\label{item:additivity} If \( x_{i} \in \mathcal{A}_{n_{i}}(\epsilon,d,R) \), \( 1 \le i \le k
    \), are such that
    \[ \Bigl| \sum_{i=1}^{k}(x_{i} - n_{i}d) \Bigr| < \epsilon, \]
    then \( \sum_{i=1}^{k} x_{i} \in \mathcal{A}_{\sum_{i=1}^{k}n_{i}}(\epsilon,d,R)\).
  \item\label{item:inclusion} If \( d \in R \) and \( m \le n \), then
    \( \mathcal{A}_{m}(\epsilon,d,R) + (n-m)d \subseteq \mathcal{A}_{n}(\epsilon,d,R) \).
  \end{enumerate}
\end{lemma}

\begin{lemma}[see Lemma 6.8 in \cite{slutsky2015}]
  \label{lem:generating-elements-of-An}
  Let \( \epsilon > 0 \), let \( 0 < \delta \le \epsilon \), and let \( x, y \in
  \mathcal{A}_{m}(\epsilon, d, R) \),
  \( m \ge 1 \), be given.  Set \( a = x - md \) and \( b = y - md \).  Suppose that \( d \in R \),
  and \( a < 0 < b \).
  There exists \( N = N_{\Lem \ref{lem:generating-elements-of-An}}(R, m, \epsilon, \delta, d, x, y) \)
  such that for all \( n \ge N \)
  \begin{itemize}
  \item if \( \delta > \gcd(a,b) \), then the set \( \mathcal{A}_{n}(\epsilon,d,R) \) is
    \( \delta \)-dense in \( \mathcal{U}_{\epsilon}(nd\+) \);
  \item if \( \delta \le \gcd(a,b) \), then the set \( \mathcal{A}_{n}(\epsilon,d,R) \) is
    \( \kappa \)-dense in \( \mathcal{U}_{\epsilon}(nd\+) \) for any \( \kappa > \gcd(a,b) \) and
    moreover
    \[ nd + k\gcd(a,b) \in \mathcal{A}_{n}(\epsilon,d,R)\ \textrm{ for all integers } k \textrm{
      such that } nd + k\gcd(a,b) \in \mathcal{U}_{\epsilon}(nd\+). \]
  \end{itemize}
\end{lemma}

Let us now explain the meaning of sets \( \tileable_{m} \) and \( \tileable_{m}^{*} \) defined
above.  We work with sets \( R_{i} \) that are subsets of
\[ \tileable_{\infty} = \tileable(\{\upsilon + t_{i} : i\in \mathbb{N}\}). \]
The problem is that there are too many possibilities for the sets \( R_{i} \), while the argument
for Lemma \ref{lem:delta-density-general} below relies upon having only finitely many possibilities
for \( R_{i} \).  So, we stratify \( \tileable_{\infty} \) into sets \( \tileable_{m} \) and note
that for any \( D > 0 \) the set \( \tileable_{m} \cap [0, D] \) is finite.  While sets \( R_{i} \)
will be infinite, each of them will be determined by a finite subset of \( \tileable_{m} \) and a
natural parameter \( r \in \mathbb{N} \).  This will let us reduce the amount of possibilities for
\( R_{i} \) to a finite number.  The exact definition is as follows.  We say that
\( R \subseteq \mathcal{U}_{\epsilon}(d\+) \) is \emph{\( r \)-tamely \( \delta \)-dense in
  \( \mathcal{U}_{\epsilon}(d\+) \)} if there exists a finite set
\( L \subseteq \mathcal{U}_{\epsilon}(d\+) \) satisfying
\begin{itemize}
\item \( R = \bigl(L + (t_{m})_{m=r}^{\infty}\bigr) \cap \mathcal{U}_{\epsilon}(d\+) \);
\item \( L \subseteq \tileable_{r}^{*} \);
\item \( L \) is \( \delta \)-dense in \( \mathcal{U}_{\epsilon}(d\+) \).
\end{itemize}
We say that \( R \) is \emph{tamely \( \delta \)-dense in \( \mathcal{U}_{\epsilon}(d\+) \)} if it is
\( r \)-tamely \( \delta \)-dense in \( \mathcal{U}_{\epsilon}(d\+) \) for some
\( r \in \mathbb{N} \).  Note that if a finite \( L \subseteq \tileable_{r}^{*} \) is
\( \delta \)-dense in \( \mathcal{U}_{\epsilon}(d\+) \), then there exists \( m_{0} \in \mathbb{N} \)
so large that for any \( m \ge m_{0} \) the set \( L + t_{m} \) is also a subset of
\( \mathcal{U}_{\epsilon}(d\+) \) and is \( \delta \)-dense in \( \mathcal{U}_{\epsilon}(d\+) \).

\begin{lemma}[cf.~Lemma 6.10 in \cite{slutsky2015}]
  \label{lem:delta-density-constant}
  For any \( \epsilon > 0 \), any \( 0 < \delta \le \epsilon \), any \( d \), any 
  \( R \subseteq \mathcal{U}_{\epsilon}(d\+) \) such that \( d \in R \) and \( R \) is tamely
  \( \epsilon \)-dense in \( \mathcal{U}_{\epsilon}(d\+) \) there exist
  \( N = N_{\Lem \ref{lem:delta-density-constant}}(\epsilon, \delta, d, R) \) and
  \( M = M_{\Lem \ref{lem:delta-density-constant}}(\epsilon, \delta, d, R) \) such that for any
  \( n \ge N \) the set \( \mathcal{A}_{n}(\epsilon, d, R) \) contains a subset that
  is \( M \)-tamely \( \delta \)-dense in \( \mathcal{U}_{\epsilon}(nd\+) \).
\end{lemma}
\begin{proof}
  Let \( r \in \mathbb{N} \) be such that \( R \) is \( r \)-tamely \( \epsilon \)-dense in
  \( \mathcal{U}_{\epsilon}(d\+) \), and pick
  \( L \subseteq \mathcal{U}_{\epsilon}(d\+) \cap \tileable_{r}^{*} \)
  witnessing this; in particular 
  \[ R = \bigl(L + (t_{m})_{m=r}^{\infty}\bigr) \cap \mathcal{U}_{\epsilon}(d\+). \]
  Since \( L \) is \( \epsilon \)-dense in \( \mathcal{U}_{\epsilon}(d\+) \), we may pick two elements
  \( l^{-}\mkern-10mu,\mkern10mu l^{+} \in L \) such that \( l^{-} < d < l^{+} \).  Using
  \( t_{m} \to 0 \), one may find sufficiently large \( m_{1} \) and \( m_{2} \) such that setting
  \( x = l^{-} + t_{m_{1}} \) and \( y  = l^{+} + t_{m_{2}} \) one has
  \begin{itemize}
  \item \( m_{1}, m_{2} \ge r \);
  \item \( x < d \) and \( y > d \);
  \item \( x, y \in \mathcal{U}_{\epsilon}(d\+) \) (thus \( x, y \in R \));
  \item \( \gcd(x - d, y - d\+) < \delta \).
  \end{itemize}
  Set \( a = x - d \) and \( b = y - d \); we have \( a < 0 < b \).  Let
  \( \tilde{R} = \{d, x, y\} \), note that
  \( \mathcal{A}_{1}(\epsilon,d,\tilde{R}) = \tilde{R} \) and Lemma
  \ref{lem:generating-elements-of-An} when applied to this \( \tilde{R} \), \( m=1 \),
  \( \epsilon \), \( \delta/2 \), \( d \) \( x \), and \( y \) produces
  \[ \tilde{N} = \tilde{N}_{\Lem \ref{lem:generating-elements-of-An}}(\tilde{R}, 1, \epsilon,
  \delta/2, d, x, y) \]
  such that for any \( n \ge \tilde{N} \) the set \( \mathcal{A}_{n}(\epsilon, d, \tilde{R}) \) is
  \( \delta/2 \)-dense in \( \mathcal{U}_{\epsilon}(nd\+) \).  
  Since \( \tilde{R} \), \( x \), and \( y \) are themselves functions\footnote{Recall that the
    sequence \( (t_{m})_{m=0}^{\infty} \) is fixed throughout the section, so dependence upon this
    sequence is ignored.} of \( R \), \( \epsilon \),
  \( \delta \), and \( d \), we have
  \( \tilde{N} = \tilde{N}(\epsilon, \delta, d, R) \).  Set \( N = \tilde{N} + 1 \) and
  \[ \bar{L} = \bigl(\{l^{-}\mkern-10mu,\mkern10mu l^{+}\} + \mathcal{A}_{\tilde{N}}(\epsilon, d,
  \tilde{R})\bigr) \cap \mathcal{U}_{\epsilon}(Nd\+). \]
  Note that
  \begin{itemize}
  \item \( \bar{L} \) is finite;
  \item \( \bar{L} \) is \( \delta \)-dense in \( \mathcal{U}_{\epsilon}(Nd\+) \), because for any
    \( \mathcal{U}_{\delta/2}(z) \subseteq \mathcal{U}_{\epsilon}(Nd\+) \)
    \[ \textrm{either }\ (z - \delta/2, z) - l^{-} \subseteq \mathcal{U}_{\epsilon}(\tilde{N}d\+)\
    \textrm{ or }\  (z, z + \delta/2) - l^{+} \subseteq \mathcal{U}_{\epsilon}(\tilde{N}d\+).  \]
  \end{itemize}
  For \( M \ge \max\{r, m_{1}, m_{2}\} \) we have \( \tilde{R} \subseteq \tileable_{M} \).  One has
  \[ \bar{L} \subseteq L + \mathcal{A}_{\tilde{N}}(\epsilon,d, \tilde{R}) \subseteq
  \tileable_{r}^{*} + \tileable_{M} \subseteq \tileable_{M}^{*}. \]
  Since \( \bar{L} \) is finite, by increasing \( M \) if necessary, we may also assume that
  \[ \bar{R} := \bar{L} + (t_{m})_{m=M}^{\infty} \subseteq \mathcal{U}_{\epsilon}(Nd\+)\ \textrm{
    and }\
  \{l^{-}\mkern-10mu,\mkern10mu l^{+}\} + (t_{m})_{m=M}^{\infty} \subseteq
  \mathcal{U}_{\epsilon}(d\+) . \]
  This guarantees that \( \bar{R} \subseteq \mathcal{A}_{N}(d, \epsilon, R) \).  Indeed, any \( z
  \in \bar{R} \) is of the form 
  \[ z = l^{\pm} + t_{m} + x\ \textrm{ for some }\ x \in \mathcal{A}_{\tilde{N}}(\epsilon, d,
  \tilde{R}) \textrm{ and } m \ge M. \]
  Since \( l^{\pm} + t_{m} \in R \) (because \( M \ge r \) and \( R \) is \( r \)-tamely
  \( \epsilon \)-dense), and since
  \[
  \mathcal{A}_{\tilde{N}}(\epsilon, d, \tilde{R}) \subseteq \mathcal{A}_{\tilde{N}}(\epsilon, d,
  R),\]
  item \eqref{item:additivity} of Lemma \ref{lem:additivity-of-An} applies, and we conclude that \(
  \bar{R} \subseteq \mathcal{A}_{N}(\epsilon,d,R) \).

  We claim these \( M \) and \( N \) satisfy the conclusion of the lemma.  The set \( \bar{R} \) is
  an \( M \)-tamely \( \delta \)-dense in \( \mathcal{U}_{\epsilon}(Nd\+) \) subset of
  \( \mathcal{A}_{N}(d, \epsilon, R) \).  Since \( d \in R \), by Lemma \ref{lem:additivity-of-An}
  one has \( \mathcal{A}_{n-1}(d,\epsilon, R) + d \subseteq \mathcal{A}_{n}(\epsilon, d, R) \), and
  so \( \bar{R} + (n-N)d \) is an \( M \)-tamely \( \delta \)-dense in
  \( \mathcal{U}_{\epsilon}(nd\+) \) subset of \( \mathcal{A}_{n}(\epsilon,d, R) \) for all
  \( n \ge N \).
\end{proof}

\begin{lemma}[cf.~Lemma 6.12 in \cite{slutsky2015}]
  \label{lem:delta-density-general}
  For any \( 0 < \epsilon \le 1 \), any \( 0 < \delta \le \epsilon \), any
  \( D > 0 \), and any \( r \in \mathbb{N} \) there exist
  \( N = N_{\Lem \ref{lem:delta-density-general}}(\epsilon, \delta, D, r) \) and
  \( M = M_{\Lem \ref{lem:delta-density-general}}(\epsilon, \delta, D, r) \) such that for any
  \( n \ge N \), any reals \( d_{i} \) and families
  \( R_{i} \subseteq \mathcal{U}_{\epsilon}(d_{i}) \), \( 1 \le i \le n \), satisfying
  \begin{itemize}
  \item \( 2\epsilon < d_{i} \le D \);
  \item \( R_{i} \) is \( r_{i} \)-tamely \( \epsilon/12 \)-dense in \( \mathcal{U}_{\epsilon}(d_{i})
    \) for some \( r_{i} \le r \);
  \end{itemize}
  the set \( \mathcal{A}_{n}\bigl( \epsilon, (d_{i})_{i=1}^{n}, (R_{i})_{i=1}^{n}\bigr) \) contains
  a subset that is \( M \)-tamely \( \delta \)-dense in \(
  \mathcal{U}_{\epsilon/2}\bigl(\sum_{i=1}^{n}d_{i}\bigr) \).
\end{lemma}

\begin{proof}
  First of all, without loss of generality we may assume that \( r \) is so big that
  \( |t_{m}| < \epsilon/12 \) for all \( m \ge r \).  Note that for any given \( r' \) and \( D' > 0
  \) sets 
  \[ \tileable_{r'}^{*} \cap [0, D'] \textrm{ and } \tileable_{r'} \cap [0,D'] \textrm{ are
    finite}, \]
  so there are only finitely many possibilities to choose a subset
  \( L \subseteq \tileable_{r'}^{*} \cap [0,D'] \) and an element
  \( d \in \tileable_{r'} \cap [0, D'] \).  This implies that there are only finitely many pairs
  \( (d,R) \) satisfying
  \begin{itemize}
  \item \( d \le D+1 \);
  \item \( d \in \tileable_{r} \);
  \item \( R \subseteq \mathcal{U}_{3\epsilon/4}(d\+) \) is \( r \)-tamely \( \epsilon/12 \)-dense in
    \( \mathcal{U}_{3\epsilon/4}(d\+) \);
  \item \( d \in R \);
  \end{itemize}
  Let \( \mathcal{Q} \) denote the set of all pairs \( (d,R) \) satisfying the conditions above.
  We set
  \[ M = \mkern-6mu\max_{(d,R) \in \mathcal{Q}}\mkern-10mu M_{\Lem
    \ref{lem:delta-density-constant}}(3\epsilon/4,\delta,d,R) \quad \textrm{and} \quad N =
  |\mathcal{Q}| \cdot\mkern-10mu \max_{(d,R) \in \mathcal{Q}} \mkern-10mu N_{\Lem
    \ref{lem:delta-density-constant}}(3\epsilon/4,\delta,d,R), \]
  and claim that these \( N \) and \( M \) work.  Let \( n \ge N \) and \( d_{i} \), \( R_{i} \),
  \( 1 \le i \le n \), be given.

  Our plan is to alter \( d_{i} \) to \( \tilde{d}_{i} \) and then apply the pigeon-hole principle
  together with Lemma \ref{lem:delta-density-constant}.
  Let
  \( L_{i} \subseteq \mathcal{U}_{\epsilon}(d_{i}) \cap \mathcal{T}_{r_{i}}^{*} \) be such that
  \[ R_{i} = \bigl(L_{i} + (t_{m})_{m=r_{i}}^{\infty}\bigr) \cap \mathcal{U}_{\epsilon}(d_{i}). \]
  Note that since \( L_{i} \) is \( \epsilon/12 \)-dense in \( \mathcal{U}_{\epsilon}(d_{i}) \), for
  any \( i \) we may pick \( l_{1}, l_{2} \in L_{i} \) such that
  \[ d_{i} - \epsilon/6 < l_{1} < d_{i} - \epsilon/12 \textrm{ and } d_{i} + \epsilon/12 < l_{2} <
  d_{i} + \epsilon/6.\] Since \( |t_{r}| < \epsilon/12 \), this ensures
  \[ d_{i} - \epsilon/4 < l_{1} + t_{r} < d_{i} \textrm{ and } d_{i} < l_{2} + t_{r} < d_{i} +
  \epsilon/4.\]
  In other words, for any \( i \) we may pick elements 
  \[ x_{1} = l_{1} + t_{r} \in R_{i} \cap \tileable_{r} \textrm{ and }  x_{2} = l_{2} + t_{r} \in
  R_{i} \cap \tileable_{r} \] which are \( \epsilon/4 \)-close to \( d_{i} \) and are
  below \( d_{i} \) and above \( d_{i} \) respectively.

  Using this observation, the construction of \( \tilde{d}_{i} \) is simple.  For \( \tilde{d}_{1} \)
  we pick any element of \( R_{1} \cap \tileable_{r} \) which is \( \epsilon/4 \)-close to \( d_{1} \). If
  \( \tilde{d}_{k} \) has been chosen, we pick \( \tilde{d}_{k+1} \) to satisfy
  \begin{itemize}
  \item \( \tilde{d}_{k+1} \in R_{k+1} \cap \tileable_{r} \);
  \item \( \bigl| \tilde{d}_{k+1} - d_{k+1}  \bigr| < \epsilon/4 \);
  \item if \( \sum_{i=1}^{k} (\tilde{d}_{i} - d_{i}) < 0 \) we want \( \tilde{d}_{k+1} > d_{k+1} \),
    and we take \( \tilde{d}_{k+1} < d_{k+1} \) otherwise.
  \end{itemize}
  The resulting sequence \( \tilde{d}_{k} \) ensures that
  \[ \Bigl| \sum_{i=1}^{k}(\tilde{d}_{i} - d_{i}) \Bigr| < \epsilon/4 \quad \textrm{ holds for all
  } k \le n. \]
  Now set \( \tilde{L}_{i} = \mathcal{U}_{3\epsilon/4}\bigl(\tilde{d}_{i}\bigr) \cap L_{i} \) and let
  \[ \tilde{R}_{i} = \bigl(\tilde{L}_{i} + (t_{m})_{m=r}^{\infty}\bigr) \cap
  \mathcal{U}_{3\epsilon/4}(\tilde{d}\mkern1mu) \subseteq \mathcal{U}_{\epsilon}( d_{i} ). \]
  A typical location of \( \tilde{d}_{i} \) relative to \( d_{i} \) is depicted in Figure
  \ref{fig:relative-location}.  Note that \( (\tilde{d}_{i}, \tilde{R}_{i}) \in \mathcal{Q} \) and
  \[ \mathcal{A}_{n}\bigl( 3\epsilon/4,(\tilde{d}\mkern1mu)_{i=1}^{n}, (\tilde{R}_{i})_{i=1}^{n} \bigr)
  \subseteq \mathcal{A}_{n}\bigl( \epsilon, (d_{i})_{i=1}^{n}, (R_{i})_{i=1}^{n} \bigr). \]
  By the choice of \( N \) and the pigeon-hole principle, there must be indices
  \[ 1 \le k_{1} < k_{2} < \ldots < k_{\tilde{N}} \le n \]
  such that \( \tilde{d}_{k_{i}} = \tilde{d}_{k_{j}} =: \tilde{d} \) and
  \( \tilde{R}_{k_{i}} = \tilde{R}_{k_{j}} =: \tilde{R} \) for all \( 1 \le i,j \le \tilde{N} \) and
  \( \tilde{N} \ge N_{\Lem \ref{lem:delta-density-constant}}(\epsilon, \delta, \tilde{d}, \tilde{R})
  \).  Since \( \tilde{d}_{i} \in \tilde{R}_{i} \), any element of
  \( \mathcal{A}_{\tilde{N}}(3\epsilon/4, \tilde{d}, \tilde{R}) \) naturally corresponds to an
  element of
    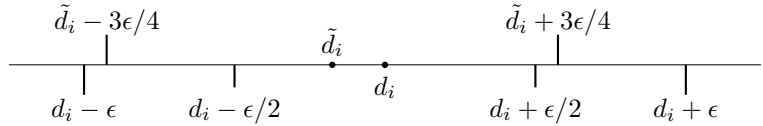
\begin{figure}[htb]
    \centering
    \begin{tikzpicture}
      \draw[very thin] (-1,0) -- (9,0);
      \draw[thick] (0,0.0) -- (0,-0.4);
      \filldraw (4,0) circle (1pt);
      \draw[thick] (8,0.0) -- (8,-0.4);
      \draw (0,-0.6) node {\( d_{i} - \epsilon \)};
      \draw (8,-0.6) node {\( d_{i} + \epsilon \)};
      \draw (4,-0.3) node {\( d_{i} \)};
      \filldraw (3.3,0) circle (1pt);
      \draw (3.3,+0.3) node {\( \tilde{d}_{i} \)};
      \draw[thick] (0.3, 0.4) -- (0.3,-0.0);
      \draw[thick] (6.3, 0.4) -- (6.3,-0.0);
      \draw (0.3,0.6) node {\( \tilde{d}_{i} - 3\epsilon/4 \)};
      \draw (6.3,0.6) node {\( \tilde{d}_{i} + 3\epsilon/4 \)};
      \draw[thick] (2,0.0) -- (2,-0.3);
      \draw[thick] (6,0.0) -- (6,-0.3);
      \draw (2,-0.6) node {\( d_{i} - \epsilon/2 \)};
      \draw (6,-0.6) node {\( d_{i} + \epsilon/2 \)};
    \end{tikzpicture}
    \caption{Location of \( \tilde{d}_{i} \) relative to \( d_{i} \).}
    \label{fig:relative-location}
  \end{figure}
  \( \mathcal{A}_{n} \bigl(\epsilon, (\tilde{d}_{i})_{i=1}^{n},
  (\tilde{R}_{i})_{i=1}^{n} \bigr) \): an element
  \( x \in \mathcal{A}_{\tilde{N}}(3\epsilon/4, \tilde{d}, \tilde{R}) \) of the form
  \( x = \sum_{i=1}^{\tilde{N}} x_{i} \), \( x_{i} \in \tilde{R} \), corresponds to
  \( y = \sum_{j=1}^{n} y_{j} \) given by
  \begin{displaymath}
    y_{j} =
    \begin{cases}
      x_{i} & \textrm{if } j = k_{i};\\
      \tilde{d}_{j} & \textrm{ otherwise}. 
    \end{cases}
  \end{displaymath}
  By the choice of \( \tilde{N} \) the set
  \( \mathcal{A}_{\tilde{N}}(3\epsilon/4, \tilde{d}, \tilde{R}) \) contains a subset which is
  \( M \)-tamely \( \delta \)-dense in \( \mathcal{U}_{3\epsilon/4}\bigl(\tilde{N}\tilde{d}\bigr) \)
  and therefore
  \( \mathcal{A}_{n} \bigl(3\epsilon/4, (\tilde{d}_{i})_{i=1}^{n}, (\tilde{R}_{i})_{i=1}^{n} \bigr)
  \) has a subset that is \( M \)-tamely \( \delta \)-dense in
  \( \mathcal{U}_{3\epsilon/4}\bigl(\sum_{i=1}^{n} \tilde{d_{i}}\bigr) \).  Finally,
  \[ \mathcal{A}_{n}\bigl(3\epsilon/4, (\tilde{d}_{i})_{i=1}^{n}, (\tilde{R}_{i})_{i=1}^{n}\bigr)
  \subseteq \mathcal{A}_{n}\bigl(\epsilon, (d_{i})_{i=1}^{n}, (R_{i})_{i=1}^{n}\bigr) \textrm{ and }
  \mathcal{U}_{\epsilon/2}\bigl(\sum_{i=1}^{n} d_{i}\bigr) \subseteq
  \mathcal{U}_{3\epsilon/4}\bigl(\sum_{i=1}^{n} \tilde{d}_{i}\bigr), \]
  implying that \( \mathcal{A}_{n}\bigl(\epsilon, (d_{i})_{i=1}^{n}, (R_{i})_{i=1}^{n}\bigr) \) contains
  a subset which is \( M \)-tamely \( \delta \)-dense in
  \( \mathcal{U}_{\epsilon/2}\bigl(\sum_{i=1}^{n} d_{i}\bigr) \) as desired.
\end{proof}

\begin{theorem}
  \label{thm:argitrarily-large-regular-blocks}
  Let \( S = \{ \upsilon + t_{m} : m \in \mathbb{N} \} \).  Any free Borel flow admits a
  cross section \( \mathcal{C} \) that has arbitrarily large \( \oer[S]{\mathcal{C}} \)-classes
  within every orbit.
\end{theorem}

\begin{proof}
  Let \( \mff \) be a free Borel flow on a standard Borel space.  Set
  \[ \epsilon_{n} = 2^{-n-1}\frac{\min\{ 1, \upsilon + t_{0}, \upsilon\}}{3}, \quad n \in \mathbb{N}.\]
  Since \( S \) generates a dense subgroup of \( \mathbb{R} \), by Lemma
  \ref{lem:dense-subgroup-asymp-dense-semigroup} we may find
  \[ K_{0} > \max\{\upsilon, \upsilon+t_{0}\} + 1 \]
  and \( M_{0} \) so big that \( \tileable_{M_{0}}^{*} \) is \( \epsilon_{0}/12 \)-dense in
  \( [K_{0}-1,\infty) \).

  Let \( \mathcal{C}_{0} \) be a cross section of \( \mff \) such that
  \( \rgap[\mathcal{C}_{0}](x) \in [K_{0} + 1, K_{0}+2] \) for all \( x \in \mathcal{C}_{0} \) (it
  exists by \cite[Corollary 2.3]{slutsky2015}).  Note that \( \oer[S]{\mathcal{C}_{0}} \) is the
  trivial equivalence relation, since \( S \subseteq [0, K_{0}] \).  Set \( D_{0} = K_{0} + 3 \),
  \( N_{0} = 1 \), and
  \begin{displaymath}
    \begin{aligned}
      N_{n+1} &= N_{\Lem \ref{lem:delta-density-general}}(\epsilon_{n}, \epsilon_{n+1}/12, D_{n},
      M_{n})\\
      M_{n+1} &= M_{\Lem \ref{lem:delta-density-general}}(\epsilon_{n}, \epsilon_{n+1}/12, D_{n},
      M_{n})\\
      D_{n+1} &= (2N_{n+1}+2)D_{n}.
    \end{aligned}
  \end{displaymath}

  We now construct cross sections \( \mathcal{C}_{n} \) inductively as follows.  We begin by
  selecting a sub cross section of \( \mathcal{C}_{0} \) which consists of pairs of adjacent points
  in \( \mathcal{C}_{0} \) with at least \( N_{1} \) at most \( 2N_{1} + 1 \) points between any two
  pairs.  By the choice of \( K_{0} \), within each pair we may move the right point by at most \(
  \epsilon_{0} \) so as to make the gap an element of \( \tileable_{M_{0}} \).  This means that we
  can add points into the resulting gap so that the distances between adjacent points will be elements
  of \( S \).  This concludes the construction of \( \mathcal{C}_{1} \).  The process is illustrated
  in Figure \ref{fig:cosparse-constructing-C1}.
  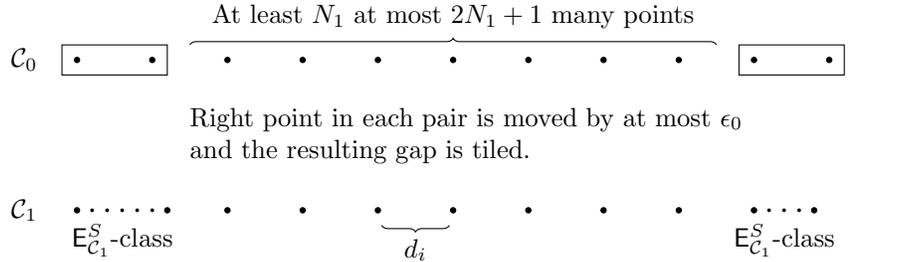
\begin{figure}[htb]
    \centering
    \begin{tikzpicture}
      \draw (-0.7,0) node {\( \mathcal{C}_{0} \)};
      \foreach \x in {0, 1, ..., 10, 11} {
        \filldraw (\x, 0) circle (1pt);
      }
      \draw (-0.2, -0.2) rectangle (1.2,0.2);
      \draw (8.8, -0.2) rectangle (10.2,0.2);
      \draw[decoration={
        brace,
        raise=0.2cm
      },
      decorate
      ] (1.5,0) -- (8.5,0) node [pos=0.5,anchor=south,yshift=0.3cm]{At least \( N_{1} \) at most \(
        2N_{1} + 1 \) many points};
      \draw (-0.7,-2) node {\( \mathcal{C}_{1} \)};
      \foreach \x in {0, 1.2, 2, 3, ..., 9, 9.8, 11} {
        \filldraw (\x, -2) circle (1pt);
      }
      \draw (5.5, -1) node[text width=8cm] {Right point in each pair is moved by at most \(
        \epsilon_{0} \) and the resulting gap is tiled.};
      \foreach \x in {0,0.2,...,1.2, 9, 9.2,..., 9.8} {
        \filldraw (\x, -2) circle (0.5pt);
      }
      \draw[decoration={
        brace,
        raise=0.2cm,
        mirror
      },
      decorate
      ] (4.05,-2) -- (4.95,-2) node [pos=0.5,anchor=south,yshift=-8mm]{\( d_{i} \)};
      \draw (0.6, -2.4) node {\( \oer[S]{\mathcal{C}_{1}} \)-class};
      \draw (9.4, -2.4) node {\( \oer[S]{\mathcal{C}_{1}} \)-class};
    \end{tikzpicture}
    \caption{Construction of \( \mathcal{C}_{1} \)}
    \label{fig:cosparse-constructing-C1}
  \end{figure}

  We call \( \oer[S]{\mathcal{C}_{1}} \)-classes, constructed via ``tiling the gap'' process, rank
  \( 1 \) blocks, and we refer to ``isolated points'' in \( \mathcal{C}_{1} \) as to rank \( 0 \)
  blocks.  It is now a good time to explain the choice of \( N_{1} \), \( M_{1} \), and \( D_{1} \).
  First of all, \( D_{0} \) represents an upper bound on the distance between adjacent points in
  \( \mathcal{C}_{0} \).  \( D_{0} \) was taken with an excess to ensure that it remains a bound
  even if each point is moved by at most \( \sum \epsilon_{k} \).  \( D_{1} \) respectively
  represents an upper bound on the distance between adjacent rank \( 1 \) blocks in
  \( \mathcal{C}_{1} \).  Between any two adjacent rank \( 1 \) blocks there are at least
  \( N_{1} \)-many rank \( 0 \) blocks, and therefore there are at least \( N_{1} \)-many gaps of
  size at least \( K_{0} \) each.  Let \( d_{1}, \ldots, d_{n} \) denote the lengths of these gaps
  (see Figure \ref{fig:cosparse-constructing-C1}).  By the choice of \( N_{1} \) and Lemma
  \ref{lem:delta-density-general}, each \( d_{i} \) can be
  distorted by at most \( \epsilon_{0} \) into \( \tilde{d}_{i} \) in such a way that
  \( \tilde{d}_{i} \in \tileable_{M_{0}} \) and the whole sum \( \sum_{i=1}^{n}d_{i} \) is distorted
  by at most \( \epsilon_{1}/12 \).  In fact, we have many ways of doing so.  To be more specific, let
  \[ L_{i} = \mathcal{U}_{\epsilon}(d_{i}) \cap \tileable_{M_{0}}^{*} \textrm{ and } R_{i} =
  \bigl(L_{i} + (t_{m})_{m=M_{0}}^{\infty}\bigr) \cap \mathcal{U}_{\epsilon}(d_{i}).  \]
  The sets \( R_{i} \) satisfy the assumptions of Lemma \ref{lem:delta-density-general} and the set
  \( \mathcal{A}_{n}\bigl(\epsilon, (d_{i})_{i=1}^{n}, (R_{i})_{i=1}^{n}\bigr) \) corresponds to
  possible ways of moving the right rank \( 1 \) block in Figure \ref{fig:cosparse-constructing-C1},
  when each rank \( 0 \) point in the midst is moved according to \( R_{i} \).  By the conclusion of
  Lemma \ref{lem:delta-density-general}, there is a set \( \bar{R} \) which is \( M_{1} \)-tamely
  \( \epsilon_{1}/12 \)-dense in
  \( \mathcal{U}_{\epsilon_{0}/2}\bigl(\sum_{i=1}^{n}d_{i}\bigr) =
  \mathcal{U}_{\epsilon_{1}}\bigl(\sum_{i=1}^{n}d_{i}\bigr) \).  To each pair of adjacent rank
  \( 1 \) blocks we associate such a set \( \bar{R} \), and during the next step of the
  construction we shall move rank \( 1 \) blocks only as prescribed by \( \bar{R} \).

  \begin{figure}[hbt]
    \centering
    \begin{tikzpicture}
      \foreach \x in {0,0.5, ..., 13}{
        \filldraw (\x, 0) circle (1pt);
        \filldraw (\x, -2) circle (1pt);
      }
      \foreach \x in {0,2,...,12}{
        \foreach \y in {0,0.1,...,0.5}{
          \filldraw (\x+\y,0) circle (0.5pt);
          \filldraw (\x+\y,-2) circle (0.5pt);
        }
      }
      \draw (-0.2, -0.2) rectangle (2.7,0.2);
      \draw (9.8, -0.2) rectangle (12.7,0.2);
      \foreach \x in {0, 0.1, ..., 2.5, 10,10.1,...,12.5}{
        \filldraw (\x,-2) circle (0.5pt);
      }
      \draw[decoration={
        brace,
        raise=0.2cm
      },
      decorate
      ] (3.5,0) -- (9.0,0) node [pos=0.5,anchor=south,yshift=0.3cm]{At least \( N_{2} \) at most \(
        2N_{2} + 1 \)-many rank \( 1 \) blocks};
      \draw (6.5, -1) node[text width=10cm] {Moving each rank \( 1 \) block by at most \(
        \epsilon_{1} \) as prescribed by \( \bar{R} \), moving each rank \( 0 \) point in rectangles
      by at most \( \epsilon_{0} \), and tiling the gaps.};
      \draw (1.25,-2.3) node {rank \( 2 \) block};
      \draw (11.25,-2.3) node {rank \( 2 \) block};
    \end{tikzpicture}
    \caption{Construction of \( \mathcal{C}_2 \)}
    \label{fig:cosparse-constructing-C2}
  \end{figure}
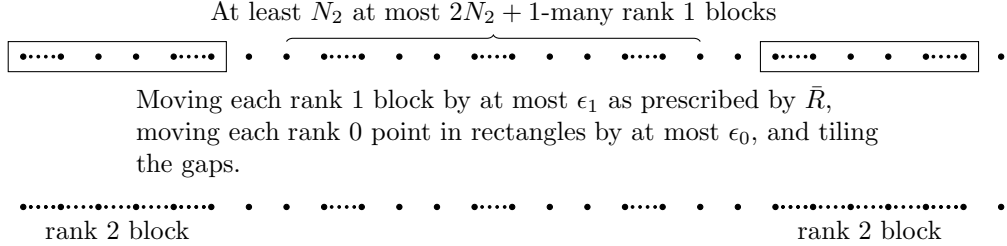

  The construction of \( \mathcal{C}_{2} \) from \( \mathcal{C}_{1} \) is analogous to the base
  step.  We pick pairs of adjacent rank \( 1 \) blocks in \( \mathcal{C}_{1} \) with at least
  \( N_{2} \) at most \( 2N_{2}+1 \) rank \( 1 \) blocks in between.  Within each pair the right
  rank \( 1 \) block is moved according to (any element of) the corresponding \( \bar{R} \), which
  results in moving each rank \( 0 \) point in between according to \( R_{i} \).  All the gaps can
  now be tiled (i.e., partitioned into segments of lengths in \( S \)), resulting in a cross section
  \( \mathcal{C}_{2} \).  Since the argument in Lemma \ref{lem:delta-density-general} provides an
  algorithm for constructing the required sets, the process can be performed in a Borel way.  The
  \( \oer[S]{\mathcal{C}_{2}} \)-classes obtained by tiling gaps in \( \mathcal{C}_{1} \) are called
  rank \( 2 \) blocks (see Figure \ref{fig:cosparse-constructing-C2}).  The procedure continues in a
  similar fashion \textemdash{} to define \( \mathcal{C}_{3} \) we take sufficiently distant pairs
  of adjacent rank \( 2 \) blocks, move the right rank \( 2 \) block within each pair by at most
  \( \epsilon_{2} \) in a way that moves each rank \( 1 \) block in between by at most
  \( \epsilon_{1} \), and each rank \( 0 \) block by at most \( \epsilon_{0} \) and turns all the
  gaps within the pair into elements of \( \tileable(S) \).  We add points to tile these gaps, thus
  creating rank \( 3 \) blocks and cross section \( \mathcal{C}_{3} \).

  When passing from \( \mathcal{C}_{n} \) to \( \mathcal{C}_{n+1} \), each block of rank \( k \) is
  moved by no more than \( \epsilon_{k} \), and if any point is moved, it becomes an element of rank
  \( n+1 \) block in \( \mathcal{C}_{n+1} \).  Since \( \sum_{i=0}^{\infty} \epsilon_{n} \)
  converges, this ensures that each point ``converges to a limit'' and the required cross section
  \( \mathcal{C} \) consists of all the ``limit points''.  It is evident from the construction that
  \( \mathcal{C} \) has arbitrarily large \( \oer[S]{\mathcal{C}} \)-blocks within each orbit.  The
  formal details of defining the limit cross section are no different from those of Theorem 9.1 in
  \cite{slutsky2015} and are similar to those in Theorem
  \ref{thm:regular-S-sections-for-sparse-flows}; we therefore omit them.
\end{proof}

\hypersetup{bookmarksdepth=-2}%

\section{\wsign Cross  sections under construction}

\vspace{-7mm}
\hypertarget{construction}{}
\vspace{7mm}
\bookmark[level=section,dest=construction]{\thesection. Cross  sections under construction}
\hypersetup{bookmarksdepth}

\begin{theorem}
  \label{thm:main-theorem}
  Let \( \mff \) be a free Borel flow on a standard Borel space \( X \) and let
  \( S \subseteq \mathbb{R}^{>0} \) be a non-empty set bounded away from zero.
  \begin{enumerate}[(I)]
  \item\label{item:lambda-regular} Assume \( \langle S \+ \rangle = \lambda\mathbb{Z} \)\kern1pt,
    \( \lambda > 0 \). The flow \( \mff \) admits an \( S \)-regular cross section if and only if it
    admits a \( \{\lambda\} \)-regular cross section.
  \item\label{item:dense-regular} Assume \( \langle S \+ \rangle\) is dense in
    \( \mathbb{R} \)\kern1pt, but \( \langle S \cap [0, n] \rangle = \lambda_{n} \mathbb{Z} \)\kern1pt,
    \( \lambda_{n} \ge 0 \), for all natural \( n \in \mathbb{N} \) (we take \( \lambda_{n} = 0 \)
    if \( S \cap [0,n] \) is empty).  The flow \( \mff \) admits an \( S \)-regular cross section if
    and only if the phase space \( X \) can be partitioned into \( \mff \)-invariant Borel pieces
    (some of which may be empty)
    \[ X = \Bigl(\bigsqcup_{i=0}^{\infty} X_{i}\Bigr) \sqcup X_{\infty} \]
    such that \( \mff|_{X_{\infty}} \) is sparse and \( \mff|_{X_{i}} \) admits a
    \( \{\lambda_{i}\} \)-regular cross section.
  \item\label{item:bounded-dense-regular} Assume there is \( n \in \mathbb{N} \) such that
    \( \langle S \cap [0,n] \rangle \) is dense in \( \mathbb{R} \).  Any free flow admits an
    \( S \)-regular cross section.
  \end{enumerate}
\end{theorem}

\begin{proof}
  \eqref{item:lambda-regular}  Suppose \( \mff \) admits an \( S \)-regular cross section, say \(
  \mathcal{C} \).  Since \( \langle S \+ \rangle = \lambda \mathbb{Z} \), every element of \( S \) is a
   multiple of \( \lambda \), so we may tile all the gaps in \( \mathcal{C} \) by intervals of
   length \( \lambda \).  More precisely,
   \[ \mathcal{D} = \bigl\{ x + k\lambda : x \in \mathcal{C}, \rgap[\mathcal{C}](x) = n\lambda, 0
   \le k < n\bigr\} \]
   is a \( \{\lambda\} \)-regular cross section.

   Suppose now \( \mff \) admits a \( \{\lambda\} \)-regular cross section, say \( \mathcal{D} \).
   It is easy to check that there exists \( N \in \mathbb{N} \) such that
   \( n\lambda \in \tileable(S) \) for all \( n \ge N \). Let \( \mathcal{C}' \) be a sub cross
   section of \( \mathcal{D} \) such that \( \rgap[\mathcal{C}'](x) \ge N\lambda \) for all
   \( x \in \mathcal{C}' \).  We have that \( \rgap[\mathcal{C}'](x) \in \tileable(S) \) for all
   \( x \in \mathcal{C}' \), and so each gap in \( \mathcal{C}' \) can be tiled by intervals of
   lengths in \( S \), which results in an \( S \)-regular cross section.

   \eqref{item:dense-regular} First suppose that \( X \) admits a decomposition into invariant
   pieces of the form
   \[ X = \Bigl(\bigsqcup_{i=0}^{\infty} X_{i}\Bigr) \sqcup X_{\infty}. \]
   Since \( \langle S \+ \rangle \) is dense in \( \mathbb{R} \) and \( \mff|_{X_{\infty}} \) is
   sparse, by Theorem \ref{thm:regular-S-sections-for-sparse-flows} \( \mff|_{X_{\infty}} \) admits
   an \( S \)-regular cross section \( \mathcal{C}_{\infty} \).  By assumption, \( \mff|_{X_{i}} \)
   admits a \( \{\lambda_{i}\} \)-regular cross section, so by item \eqref{item:lambda-regular} it also
   admits an \( S\cap [0, i] \)-regular cross section \( \mathcal{C}_{i} \).  The union 
   \[ \mathcal{C}_{\infty} \sqcup \bigsqcup_{i\in \mathbb{N}} \mathcal{C}_{i} \]
   of these cross sections is an \( S \)-regular cross section on \( X \).

   For the other direction suppose \( \mathcal{C} \) is an \( S \)-regular cross section for \( \mff
   \).  Let \( X_{\infty} \) be the set of orbits where the gap function is unbounded:
   \[ X_{\infty} = \Bigl\{x \in \mathcal{C} :
   \sup\bigl\{\rgap[\mathcal{C}]\bigl(\phi^{k}_{\mathcal{C}}(x)\bigr) : k \in \mathbb{Z} \bigr\} =
   \infty \Bigr\}. \]
   Some orbits in \( X_{\infty} \) may not have ``bi-infinitely'' unbounded gaps, but the restriction
   of the flow onto the set of such orbits is smooth; we may thus modify \( \mathcal{C} \) on this set
   and assume that \( \mathcal{C} \cap X_{\infty} \) is always ``bi-infinitely'' unbounded, and is
   therefore a sparse cross section.  Thus \( \mff|_{X_{\infty}} \) is sparse.  Let for \( i \in
   \mathbb{N} \)
      \[ X_{i+1} = \Bigl\{x \in \mathcal{C} :
   \sup\bigl\{\rgap[\mathcal{C}]\bigl(\phi^{k}_{\mathcal{C}}(x)\bigr) : k \in \mathbb{Z} \bigr\} \le i+1
   \Bigr\} \setminus X_{i}, \]
   where \( X_{0} = \es \).  By assumption there is \( \lambda_{i} \in \mathbb{R}^{\ge 0} \) such
   that \( \langle S \cap [0,i] \rangle = \lambda_{i}\mathbb{Z} \).  Since all the gaps in \(
   \mathcal{C} \cap X_{i} \) belong to \( S \cap [0,i] \),  item
   \eqref{item:lambda-regular} applies, and \( \mff|_{X_{i}} \) admits a \( \{ \lambda_{i} \} \)-regular
   cross section.

   \eqref{item:bounded-dense-regular} Suppose \( \langle S \cap [0,n] \rangle \) is dense in
   \( \mathbb{R} \).  We may assume for notational convenience that \( S \) itself is bounded and
   \( \langle S \+ \rangle \) is dense in \( \mathbb{R} \).  For a bounded subset of
   \( \mathbb{R} \) to generate a dense subgroup, one of two things has to happen.  One possibility
   is that \( S \) contains two rationally independent reals \( \alpha, \beta \in S \).  If this is
   the case, Theorem 9.1 of \cite{slutsky2015} applies and generates an
   \( \{\alpha, \beta\} \)-regular cross section for \( \mff \).

   The other possibility is that there are infinitely many elements in \( S \).  In that case we may
   select a limit point \( \upsilon \) for \( S \).  While \( \upsilon \) is not necessarily an
   element of \( S \), there is a sequence \( (s_{n})_{n=0}^{\infty} \subseteq S \), which we may
   assume to be monotone, such that \( s_{n} \to \upsilon \).  We therefore find ourselves in the
   context of Theorem \ref{thm:argitrarily-large-regular-blocks} for \( t_{m} = s_{m} - \upsilon \),
   which ensures existence of a cross section \( \mathcal{D} \) with arbitrarily large
   \( \oer[S]{\mathcal{D}} \)-classes within each orbit.  Orbits in \( \mathcal{D} \) split into
   three categories:
   \( \mathcal{D} = \mathcal{D}_{r} \sqcup \mathcal{D}_{0} \sqcup \mathcal{D}_{s} \), where
   \begin{itemize}
   \item \( D_{r} \) consists from those orbits which constitute a single
     \( \oer[S]{\mathcal{D}} \)-class;
   \item \( \mathcal{D}_{0} \) contains orbits which have at least two
     \( \oer[S]{\mathcal{D}} \)-classes at least one of which is infinite;
   \item \( \mathcal{D}_{s} \) draws all the
     orbits with all \( \oer[S]{\mathcal{D}} \)-classes being finite.
   \end{itemize}
   More formally, sets \( D_{r} \), \( D_{0} \), and \( D_{s} \) are given by
   \begin{displaymath}
     \begin{aligned}
       \mathcal{D}_{r} =& \bigl\{ x \in \mathcal{D}: x\, \oer[S]{\mathcal{D}}
       \phi^{k}_{\mathcal{D}}(x) \textrm{ for all } k \in \mathbb{Z} \bigr\},\\
       \mathcal{D}_{0}^{*} =& \bigl\{ x \in \mathcal{D} : \exists k \in \mathbb{Z}\ \forall n \in
       \mathbb{N}\quad \phi^{k}_{\mathcal{D}}(x)\, \oer[S]{\mathcal{D}} \,
       \phi^{k+n}_{\mathcal{D}}(x)\bigr\} \cup\\ & \bigl\{ x \in \mathcal{D} : \exists k \in
       \mathbb{Z}\ \forall n \in \mathbb{N}\quad \phi^{k}_{\mathcal{D}}(x)\, \oer[S]{\mathcal{D}} \,
       \phi^{k-n}_{\mathcal{D}}(x)\bigr\},\\
       \mathcal{D}_{0} =& \mathcal{D}_{0}^{*} \setminus \mathcal{D}_{r},\\
       \mathcal{D}_{s} =& \bigl\{ x \in \mathcal{D} : \forall k \in \mathbb{Z}\ \exists m, n \in
       \mathbb{Z}\ (m < 0) \textrm{ and } (n > 0) \textrm { and },\\
       & \neg\bigl( \phi_{\mathcal{D}}^{k}(x)\, \oer[S]{\mathcal{D}} \phi^{k+m}_{\mathcal{D}}(x)
       \bigr) \textrm{ and } \neg\bigl( \phi_{\mathcal{D}}^{k}(x)\,
       \oer[S]{\mathcal{D}} \phi^{k+n}_{\mathcal{D}}(x) \bigr)\bigr\}.\\
     \end{aligned}
   \end{displaymath}
   Set \( X_{r} \), \( X_{0} \), and \( X_{s} \) to be the saturation of sets \( \mathcal{D}_{r} \),
   \( \mathcal{D}_{0} \), and \( \mathcal{D}_{s} \):
   \[ X_{r} = \mathcal{D}_{r} + \mathbb{R}, \qquad X_{0} = \mathcal{D}_{0} + \mathbb{R}, \qquad X_{s}
   = \mathcal{D}_{s} + \mathbb{R}. \]
   The set \( \mathcal{D}_{r} \) is an \( S \)-regular cross section for \( \mff|_{X_{r}} \).  The
   restriction of \( \mff \) onto \( X_{0} \) is smooth, as taking finite endpoints of infinite
   \( \oer[S]{\mathcal{D}_{0}} \)-classes picks at most two points from each orbit.  The flow
   \( \mff|_{X_{0}} \) therefore admits any kind of cross section.  It remains to deal with the
   restriction of \( \mff \) on \( X_{s} \).  Let \( \mathcal{C} \subseteq \mathcal{D}_{s} \) to
   consist of endpoints of \( \oer[S]{\mathcal{D}_{s}} \)-classes:
   \[ \mathcal{C} = \bigl\{ x \in \mathcal{D}_{s} : \neg\bigl( x\, \oer[S]{\mathcal{D}_{s}}\,
     \phi_{\mathcal{D}_{s}}(x)\bigr) \textrm{ or } \neg\bigl( x\, \oer[S]{\mathcal{D}_{s}}\,
     \phi^{-1}_{\mathcal{D}_{s}}(x)\bigr) \bigr\}. \]
   The condition of having arbitrarily large \( \oer[S]{\mathcal{D}_{s}} \)-classes ensures that
   \( \mathcal{C} \) is a {\it sparse\/} cross section for \( \mff|_{X_{s}} \).
   Theorem~\ref{thm:regular-S-sections-for-sparse-flows} applies and finishes the proof.
\end{proof}

\begin{remark}
  \label{rem:ergodic-version-of-main-thm}
  In the ergodic theoretical framework, when two flows that differ on a set of measure zero are
  identified, items \eqref{item:dense-regular} and \eqref{item:bounded-dense-regular} collapse.
  This is because every flow that preserves a finite measure is sparse on an invariant set of full
  measure (see \cite[Theorem 3.3]{slutsky2015}).  For an ergodic theorist any free flow admits an
  \( S \)-regular cross section whenever \( S \) generates a dense subgroup of \( \mathbb{R} \).
\end{remark}

In conclusion we would like to give an example of a flow which illustrates the difference
between items \eqref{item:dense-regular} and \eqref{item:bounded-dense-regular} above in the Borel
setting.  Let \( \sigma : 2^{\mathbb{N}} \to 2^{\mathbb{N}} \) be the odometer map: if
\( x \in 2^{\mathbb{N}} \) is such that \( x = 1^{n}0 * \), then \( \sigma(x) = 0^{n} 1 * \); also
\( \sigma(1^{\infty}) = 0^{\infty} \).  This is a free Borel automorphisms on the Cantor space.  

\begin{proposition}
  Let \( S \subseteq \mathbb{R}^{>0} \) be a non-empty set of positive reals bounded away from zero
  such that
  \begin{itemize}
  \item \( S \) generates a dense subgroup of \( \mathbb{R} \).
  \item All elements in \( \mathbb{R} \) are pairwise rationally dependent, i.e., \( \langle S\+
    \rangle = 
    \beta \mathbb{Q} \) for some \( \beta \in \mathbb{R}^{>0} \).
  \item \( S \cap [0,n] \) is finite for every \( n \in \mathbb{N} \).
  \end{itemize}
  A typical example is the set of partial sums of the harmonic series:
  \[ S = \Bigl\{ \sum_{i=1}^{n} \frac{1}{i} : n\ge 1\Bigr\}. \]
  Let \( \alpha \in \mathbb{R}^{>0} \) be any real that is rationally independent form \( \beta \).
  Let \( \mff \)  be the flow under the
  constant function \( \alpha \) over the Cantor space \( 2^{\mathbb{N}} \) with odometer
  \( \sigma \) as the base automorphism.  Such a flow does not admit an \( S \)-regular
  cross section.
\end{proposition}

\begin{figure}[htb]
  \centering
  \begin{tikzpicture}
    \draw[thick] (0,0) -- (4,0);
    \draw (0,2) -- (4,2);
    \draw (2,1) node {\( \Omega = 2^{\mathbb{N}} \times [0,\alpha) \)};
    \draw (3.5,2.2) node {\( f \equiv \alpha \)};
    \draw (3.1, 0.25) node {\( \sigma : 2^{\mathbb{N}} \to 2^{\mathbb{N}} \)};
  \end{tikzpicture}
  \caption{The flow \( \mff \) under the function \( f \equiv \alpha \) with the base automorphism
    \( \sigma \).}
  \label{fig:flow-under-alpha}
\end{figure}
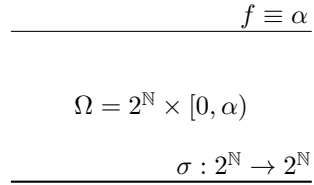

\begin{proof}
  The proof is by contradiction.  Set \( \Omega = 2^{\mathbb{N}} \times [0,\alpha) \) and suppose
  that \( \mathcal{C} \subseteq \Omega \) is an \( S \)-regular cross section for \( \mff \).  The
  space \( \Omega \) can naturally be endowed with a compact topology which turns \( \mff \) into a
  continuous flow on a compact metric space.  Indeed,
  \( \Omega = 2^{\mathbb{N}} \times [0, \alpha] /\! \sim \), where \( \sim \) identifies
  \( (x,\alpha) \) with \( \bigl(\sigma(x), 0\bigr) \).  Since \( \sim \) is closed, the factor
  topology turns \( \Omega \) into a compact metric space, and the flow \( \mff \) is seen to be
  continuous.  Moreover, since \( \sigma \) is a minimal\footnote{A homeomorphism is minimal if
    every its orbit is dense. } homeomorphism of the Cantor space, one easily checks that \( \mff \)
  on \( \Omega \) is also minimal.  By Proposition 3.2 in \cite{slutsky2015}, there is a Borel
  invariant comeager subset \( Z \subseteq \Omega \) such that \( Z \cap \mathcal{C} \) has all
  gaps bounded by some \( n_{0} \in \mathbb{N} \).  Since \( S \cap [0,n_{0}] \) is finite, and since all
  elements in \( S \cap [0,n_{0}] \) are rationally dependent, there is
  \( \lambda \in \mathbb{R}^{>0} \) such that
  \( \langle S \cap [0,n_{0}] \rangle = \lambda \mathbb{Z} \).  By item \eqref{item:lambda-regular}
  of Theorem \ref{thm:main-theorem} this means that there is a \( \{\lambda\} \)-regular cross
  section \( \mathcal{D} \subseteq Z \) for the flow \( \mff|_{Z} \).  By Ambrose's criterion for
  existence of a \( \{\lambda\} \)-regular cross section (see Proposition \ref{prop:Ambrose}) there
  exists a Borel function \( f : Z \to \mathbb{C}\setminus\{0\} \) such that
  \[ f(x + r) = e^{\textstyle \frac{2\pi i r}{\lambda}} f(x) \textrm{ for all } x \in Z \textrm{ and
    } r \in \mathbb{R}. \]
  Let \( X = \mathrm{proj}_{2^{\mathbb{N}}}(Z) \).  Since \( Z \) is \( \mff \)-invariant, 
  \[ X = Z \cap \bigl\{(x, 0) : x \in 2^{\mathbb{N}}\bigr\}. \]
  Note that \( X \) must be Borel, \( \sigma \)-invariant, and comeager in \( 2^{\mathbb{N}} \).  We
  restrict the function \( f \) to the base \( 2^{\mathbb{N}}\times \{0\} \).  Since \( f \) is
  Borel, there is a comeager subset \( \tilde{X} \) of \( 2^{\mathbb{N}} \times \{0\} \) such that
  \( f|_{\tilde{X}} \) is continuous (see \cite[8.38]{kechris_classical_1995}).  Without loss of
  generality we may assume that \( \tilde{X} \subseteq X \) and that \( \tilde{X} \) is
  \( \sigma \)-invariant.  Pick \( x_{0} \in \tilde{X} \).  Since \( x + \alpha = \sigma(x) \) for
  all \( x \in 2^{\mathbb{N}} \),
  for any \( k \in \mathbb{N} \) we have
  \[ f\bigl(\sigma^{k}(x_{0})\bigr) = f(x_{0} + k\alpha) = e^{\textstyle \frac{2 \pi i
      k\alpha}{\lambda}}f(x_{0}). \]
  We take \( k = 2^{m} \) in the above.  Since \( \sigma^{2^{m}}(x) \to x \) for all
  \( x \in 2^{\mathbb{N}} \), and since \( f \) is continuous on \( \tilde{X} \) which is
  \( \sigma \)-invariant, we get
  \[ e^{\textstyle \frac{2 \pi i 2^{m}\alpha}{\lambda}}f(x_{0}) \to f(x_{0}) \textrm{ as } m \to \infty,\]
  which is equivalent to \( 2^{m} \alpha/\lambda \to 0 \mod \mathbb{Z} \), because
  \( f(x_{0}) \ne 0 \).  By assumption on \( S \), \( \alpha/\lambda \) is an irrational number,
  thus to finish the proof it remains to show that \( 2^{m}\gamma \not \to 0 \mod \mathbb{Z} \) for
  any irrational \( \gamma \in \mathbb{R}^{>0}\).

  Suppose towards a contradiction that \( 2^{m}\gamma \to 0 \mod \mathbb{Z} \) for an irrational
  \( \gamma \).  Pick \( m_{0} \) so big that for every \( m \ge m_{0} \) there is
  \( k_{m} \in \mathbb{Z} \) such that \( |2^{m}\gamma - k_{m}| < 1/4 \).  Let
  \( a = 2^{m_{0}}\gamma - k_{m_{0}} \), and let \( p \in \mathbb{N} \) be the smallest natural such
  that \( |2^{p}a| \ge 1/4 \).  It is easy to see that \( |2^{p}a| < 1/2 \).  Therefore,
  \[ 1/4 \le \bigl| 2^{p+m_{0}}\gamma - 2^{p}k_{m_{0}} \bigr| < 1/2. \]
  Thus \( 2^{m} \not \to 0 \mod \mathbb{Z} \) as claimed.
\end{proof}

\bibliographystyle{alpha}
\bibliography{refs}

\end{document}